\documentclass[reqno,12pt]{amsart}
\usepackage{latexsym}
\usepackage{graphicx,psfrag,amssymb,amsmath,amsthm}
 \usepackage[dvips]{color}
 \usepackage{hyperref,psfrag}
\def\beq{\begin{equation}}
\def\eeq{\end{equation}}

\def\cst{{\rm Cst}}

\def\bB{\mathbf{B}}

\def\bD{\mathbf{D}}
\def\bE{\mathbf{E}}
\def\bF{\mathbf{F}}
\def\bG{\mathbf{G}}

\def\bJ{\mathbf{J}}

\def\bP{\mathbf{P}}

\def\bn{\mathbf{n}}

\def\bp{\mathbf{p}}
\def\bq{\mathbf{q}}

\def\bv{\mathbf{v}}

\def\bx{\mathbf{x}}

\def\eps{\varepsilon}

\def\bDelta{{\boldsymbol{\Delta}}}

\def\bsigma{{\boldsymbol{\sigma}}}

\def\b1{{\boldsymbol{1}}}

\def\cC{\mathcal{C}}

\def\cE{\mathcal{E}}
\def\cF{\mathcal{F}}

\def\cH{\mathcal{H}}
\def\cI{\mathcal{I}}
\def\cJ{\mathcal{J}}
\def\cK{\mathcal{K}}

\def\cM{\mathcal{M}}
\def\cN{\mathcal{N}}
\def\cO{\mathcal{O}}

\def\cR{\mathcal{R}}
\def\cS{\mathcal{S}}

\def\tq{\tilde{\bq}}

\def\tf{\tilde{f}}

\def\tT{\widetilde{T}}

\def\tpi{\tilde{\pi}}

\def\tQ{\widetilde{Q}}
\def\tM{\widetilde{M}}

\def\eps{\varepsilon}

\def\pQ{\partial Q}

\newtheorem{theorem}{Theorem}[section]

\newtheorem{lemma}[theorem]{Lemma}
\newtheorem{proposition}[theorem]{Proposition}
\newtheorem{corollary}[theorem]{Corollary}

\theoremstyle{definition}

\newtheorem{remark}{Remark}

\numberwithin{equation}{section}

\begin{document}

\title{Electrical current in Sinai billiards under general small forces}

\author{Nikolai Chernov}
\address{Department of Mathematics, University of Alabama at Birmingham, Birmingham, AL
35294}
\email{chernov@math.uab.edu}

\author{Hong-Kun Zhang}
\author{Pengfei Zhang}
\address{Department of Mathematics and Statistics, UMass Amherst, Amherst, MA 01002}
\email{hongkun@math.umass.edu, pzhang@math.umass.edu}

\date{\today}

\maketitle



\begin{abstract}
 The  Lorentz gas
of $\mathbb{Z}^2$-periodic scatterers (or the so called Sinai billiards) can be used to model motion of electrons on an ionized medal. We investigate the linear response for the system
under various external forces (during both the flight and the collision).
We give some characterizations under which the forced system is
time-reversible, and
derive an estimate of the electrical current generated by the forced system.
Moreover, applying Pesin entropy formula and Young dimension formula, we get
several characterizations of the non-equilibrium steady state of the forced system.
\end{abstract}

\section{Introduction}

Lorentz gas is a popular model in mathematical physics introduced in
1905, (see \cite{Lor}), in studying the motions of a point-particle
or a gas of particles (electrons) in a metallic conductor. Here we
consider a two-dimensional periodic Lorentz gas, that is, a particle
moves on the plane and bounces off a $\mathbb{Z}^2$-periodic ray of
scatterers (ions). In this case the dynamics reduces to a
dynamical billiard system on the 2-D torus $\mathbb{T}^2$, generated by a
billiard moving freely until it bounces off the scatterers. More
precisely, let $\bB_1,\cdots, \bB_s$ be open convex domains on
$\mathbb{T}^2$ with mutually disjoint closures, and
$Q=\mathbb{T}^2\backslash \bigcup \bB_i$. Moreover, we assume that
the boundary of each $\bB_i$ is $C^3$ smooth with non-vanishing
curvature.

The study of classical billiard dynamics was originated in 1970 by Sinai. This model is a Hamiltonian system, so it preserves the kinetic energy $E=\frac{1}{2}\|\bp\|^2$.
Therefore we can restrict the billiard flow to a 3-D submanifold $\cM_0=\{(\bq,\bp):E(\bq,\bp)=c\}$,
on which the Liouville measure is also preserved by the flow.
Moreover there is a 2-D global cross-section $M_0\subset\cM_0$, which consists of the post-collision vectors:
$M_0=\{(\bq,\bp)\in\cM_0:\bq\in\partial Q, \bp\cdot\bn(\bq)\ge0\}$. The Poincar\'e map $T_0:M_0\to M_0$
preserves a smooth measure $d\nu_0=\text{cst}\cdot\cos\varphi dr d\varphi$, where $\varphi$
is the angle from the normal vector $\bn(\bq)$ and $\bp$.
The map $T_0$ and the flow $\Phi_0$ have been shown to be uniformly hyperbolic and
Bernoulli \cite{Si70,GO74}, and many statistical properties have been well understood and  are proved, see \cite{BSC,You98,Ch99}, and the references therein.

Recently, much attention  is shifting to the investigation of statistical properties of nonequilibrium billiards.
Nonequilibrium phenomena are characterized by the action of external forces
or boundary conditions for transport  equations that change the system and generate a steady
process that can be measured--mass transfer, energy (heat) transfer, charge transfer
(electrical current), entropy production, or others. The laws in equilibrium statistical physics are better understood
and proved or almost proved in quite a few cases.
However,  the apparatus of equilibrium statistical mechanics still relies
largely on  heuristic statements based on numerical results,
and only very few models have been  studied with sufficient
mathematical rigor.  The main difficulty is that nonequilibrium  billiards have singularities and unbounded derivatives,  they usually do not preserve smooth measures, and their evolution is described by steady states characterized by singular invariant measures, of which relatively little is known in general.
In addition, since Gibbs entropy is invariant under a Hamiltonian time evolution,  the study of
entropy increase (the second law of thermodynamics) in nonequilibrium systems is far from straightforward \cite{Ru99}.
One of the first nonequilibrium physical models that were studied rigorously is the
periodic Lorentz gas with a constant electrical field by Chernov, Eyink,  Lebowitz and Sinai  \cite{CELS1,CELS2} and the famous Ohm's law was proved for that case.
Similar studies on special nonequilibrium dynamics were conducted by
Bunimovich and Spohn \cite{BS1}, Gallavotti \cite{G95}, Ruelle \cite{Ru99}, and others.

Here we investigate some physical laws for  Sinai billiards
(or periodic Lorentz gases) under general external forces. Let $\bq=(x,y)$ be the position
of a particle in a billiard table $Q:=\mathbb{T}^2\setminus \cup_i \bB_i$,
and $\bp=(\dot x, \dot y)$ be the velocity vector.
We add two types of forces to the system in the following steps:
\begin{enumerate}
\item (\textbf{during the  flight}) Let $\bF=\bF(\bq, \bp)$ be a stationary external force on $Q$.
The forced billiard flow is governed by the following differential equation between collisions:
\beq\label{flow0}
\left\{
  \begin{array}{ll}
    \dot \bq=\bp, \\
\dot \bp= \bF,
  \end{array}
\right.
\eeq
where the dot derivative refers to differentiate with respect to the time $t$.

\item (\textbf{at the moment of collision}) Let $\bG$ be an external twisting force spreading on $\pQ$
that acts on each incoming trajectory
right after its elastic collision with $\pQ$:
\begin{equation}
\label{reflectiong}
(\bq^+(t_i), \bp^+(t_i)) =\bG(\bq^-(t_i), \cR\bp^-(t_i)),
\end{equation}
 where $\cR \bp^-(t_i)= \bp^-(t_i)+2(\bn(\bq^-)\cdot \bp^{-})\bn(\bq^-))$
 is the usual elastic reflection operator,
 $\bn(\bq)$ is the unit normal vector to the billiard wall $\partial Q$ at
 $\bq$ pointing inside of the table $Q$, and $\bq^-(t_i), \bp^-(t_i)$, $\bq^+(t_i)$ and
 $ \bp^+(t_i)$ refer to the incoming and  outgoing position and velocity vectors, respectively.
\end{enumerate}
Note that the twisting force $\bG$ changes not only the outgoing velocity of the billiard,
but also the position of the billiard along the boundary $\partial Q$.
The change in velocity can be thought of as a kick, while a change in
position can model a slip along the boundary at collision.
These forces indicate that the system will experience nonelastic reflections.

The type of forces we consider here are quite common in many physical models.
For example we
can have a potential function $U$ on $Q$, or an electromagnetic field (when the moving billiard
is an electron), see \cite{CELS1,CELS2,Ch01,Ch08}.
The twisting force right after the collision is also closely related to
real-world models, especially the so called soft scatterers, see \cite{Ba,Kn}.
Recall that the unforced system can be viewed  as a Hamiltonian system under
a potential function $U$ given by $U(\bq)=0$ if $\bq\in Q$, and $+\infty$ on the  scatterers.
In this case, the scatterers $B_i$ are said to be rigid (or hard).
Now if we replace $U$ by some finite potential function on each $B_i$,
then the resulting $B_i$ is a soft one: a running billiard will climb up $B_i$ and then
exit very soon. Different shapes of the scatterer result in different enter-exit relations.
If we view what happen on the scatterers as a black box, the effect can be
understood as a twisting force $\bG$ right after the elastic collision.
Clearly this kind of twisting forces not only preserve the tangent collisions, and are also
time-reversible. In this way we get a uniform treatment of various softened scatterers.

In \cite{Ch01, Ch08}, Chernov considered billiards under small external forces $\bF$
between collisions, and proved several ergodic and statistical properties of the SRB measure for the perturbed billiard system.
 In \cite{DC} Dolgoyat and Chernov put a constant electric field on Lorentz gases with infinite horizon
and got various characterizations of the steady state electric current generated by the forced system.
The systems with some simple twist forces were considered in \cite{Z11}, assuming that
$\bG$ depends on and affects only the velocity, not the position.
The Green-Kubo type formula was proved, and it was shown that the current generated by the forced flow is closely related to the strength of the force.
Very recently, Chernov and Korepanov investigated in \cite{CK} the linear response for
Sinai billiards under external forces (without twisting).
In  this paper, we consider the dynamics of the Sinai
billiard on the table $Q$, but subject to more general forces $\bP=(\bF,\bG)$ both during flight and at collisions. We characterized certain properties of the SRB measure for the forced systems by  obtaining the Pesin's K-S entropy formula
and Young's expression for the fractal dimension. Moreover, we also prove rigorously the Green-Kubo like formula and investigate the linear response formula.\\

\noindent{\bf Structure of the paper.}  This paper is organized as follows. In \S 2 we list
the main assumptions on the external forces $\bP=(\bF,\bG)$, introduce some basic notations
and propositions of our systems. In \S 3 we state the main theorems about the linear responses and
statistical properties of our forced billiard system. We also list quite a few examples
of the external forces and study the current generated by these forced systems.
Then we divide the analysis of the forced systems  into two
steps: in \S 3 we study the effect of the force $\bF$ during the flight, and in \S 4 we add
the twisting effect and conclude the proof of our main theorems.

\section{Main results}
\subsection{Assumpsions}

In this subsection we  first state the assumptions on the model, which
combine the assumptions in \cite{Ch01, Z11}. Let $\bP=(\bF,\bG)$,
where $\bF$ and $\bG$ are the two external forces
during the flight and right after the reflection, respectively.
Let $\Phi_{\bP}$ be the induced billiard flow on $Q$. \\

\noindent\textbf{(A1)} (\textbf{Invariant space}) \emph{The
forced flow $\Phi_{\bP}$ preserve a smooth function $\cE(\bq,
\bp)$, such that the level surface $\cM:=\{\cE(\bq,\bp)=c\}$ is a
compact 3-D manifold, for some $c>0$.
Moreover, $\| \bp\|>0$ on $\cM$, and for each $\bq\in Q$ and $ \bp\in S^1$, the ray $\{(\bq, t \bp), t>0\}$ intersects the manifold $\cM$ in exactly one point.} \\

Assumption ({\textbf{A1}}) specifies an additional integral of
motion, so that we only consider the restricted systems on a compact
phase space. For example, we can add a Gaussian thermostat (a heat
bath) to the system such that the billiard moves at a constant speed
(constant temperature if there are a large number of particles).
Then $\cM:=\{\|\bp\|=c\}$ is an invariant compact level set.

Under the assumption ({\textbf{A1}}), the speed $p=\| \bp\|$ of the
billiard along any typical trajectory on $\cM$ at time $t$ satisfies
$$0<p_{\min}\leq p(t) \leq p_{\max}<\infty,$$
for some constants $p_{\min}\leq p_{\max}$. Moreover, $\cM$ admits a
global coordinate system $\{(x,y,\theta):(x,y)\in
Q,0\le\theta<2\pi\}$, where $\theta$ is the angle between $\bp$ and
the positive $x$-axis. In particular, the speed $p=\|\bp\|$ on $\cM$
can be represented as a function $p=p(x,y,\theta)$. Then the
velocity $\bp$ at $\bq$ is given by $\bp=p\bv$, where
$\bv=(\cos\theta, \sin\theta)$ is the unit vector in the
direction of $\bp$. So Eq. \eqref{flow0} of the dynamics between
collisions can be rewritten as \beq\label{vpF} \dot p
\bv+p\dot\bv=\bF. \eeq Multiplying $\bv$ to both sides of
(\ref{vpF}) using dot product and cross product respectively, we
then get \beq\label{pvF} \dot p= \bv\cdot \bF,\,\,\,\,\,\text{ and
}\,\,\,\,\,\,p\bv\times \dot\bv=\bv\times \bF. \eeq Therefore, the
equations in (\ref{flow0}) have the following coordinate
representations that, at any $(x,y,\theta)\in\cM$, \beq\label{flow}
\left\{
  \begin{array}{ll}
    \dot x=p\cos\theta, \\
    \dot y=p\sin\theta, \\
  \dot \theta=(-F_1\sin\theta+F_2\cos\theta)/p.
  \end{array}
\right.
\eeq
Consider the trajectory $\tilde\gamma\subset\cM$ of the flow passing through the point $(x,y,\theta)\in\cM$,
which projects down to a smooth curve $\gamma\subset Q$.
Let $h=h(x,y,\theta)$ be the (signed) geometric curvature of $\gamma$ at the base point $(x,y)\in Q$.
Then we have that
\beq\label{formh}
h(x,y,\theta)=\pm\frac{\|\dot\bq\times\ddot\bq\|}{\|\dot\bq\|^3}=\pm\frac{\|\bv\times \bF\|}{p^2}
=\frac{-F_1\sin\theta+F_2\cos\theta}{p^2},
\eeq
where the sign should be chosen accordingly. Then combining with (\ref{flow}), we have
\beq\label{ptheah}
\dot\theta= p h.
\eeq
Note that the angle $\theta=\theta(t)$ experiences a discontinuity at the times of reflection.
That is, it changes from $\theta^-$ to $\theta^+$.
In the elastic collision case, all other quantities ($x$, $y$ and $p$) stay the same.
For example the speed $p(x,y,\theta^+)=p(x,y,\theta^-)$
(here $(x,y)\in\partial Q$). Under the twisting forces, all quantities are subject to change.\\

For any phase point $(\bq,\theta) \in \cM$ for the flow, let
$\tau(\bq,\theta)$ be the time for the trajectory from
$(\bq,\theta)$ to its next non-tangential collision at $\partial Q$. \\

\noindent(\textbf{A2}) (\textbf{Finite horizon})
\emph{There exist $\tau_{\max}>\tau_{\min}>0$ such that free paths between successive reflections are uniformly bounded:
$\tau_{\min} \le \tau(\bq,\theta ) \le \tau_{\max}$, for all $(\bq,\theta) \in \cM$ with $\bq\in\pQ$.
In addition,  the curvature $\cK(r)$ of the boundary $\partial Q$
is also uniformly bounded for all $r \in \partial Q$.} \\

Assumption (\textbf{A2}) implies that there exists a 2-D global
cross-section, the post-collision space $M$, of the perturbed
billiard flow $(\cM,\Phi_{\bP})$: $M=\{(\bq,\bp)\in
\cM:\bq\in\partial Q, \bp\cdot\bn(\bq)\ge0\}$, which consists all
outgoing vectors in $\cM$ based at the boundary of the billiard
table $Q$. Denote by $T_{\bP}: M\to M$ the Poincar\'e map induced by
the forced flow $\Phi_{\bP}$ on $\cM$. Moreover, the 2-D space $M$ can be
parameterized by $\bx=(r,s)$, where $r$ is the arclength parameter
of $\partial  Q$ oriented clockwise, and a new parameter
$s=\sin\varphi$, where $\varphi$ is the angle formed by the outgoing
vector $\bp$ and the normal vector $\bn(\bq)$. This coordinate
system has the advantage that the Lebesgue measure $d\mu_0=\cst\,d
r\, d s$ coincides with the measure $d\nu_0=\cst \cos\varphi dr\,
d\varphi$. Using this new coordinate system, the collision time can
be written as $\tau(\bx)=\tau(r,s)$, and the twist force $\bG$ at
the collision can be reformulated as $(\bar r,\bar s)=\bG(r,s)
=(r,s)+(g^1(r,s),g^2(r,s))$. \\

\noindent(\textbf{A3}) (\textbf{Smallness of the external forces}). \emph{There exists
$\eps>0$ small enough, such that the forces $\bP=(\bF,\bG)$ satisfy
$$\|{\bF}\|_{C^1}<\eps, \|\bG-\mathrm{Id}_M\|_{C^1}<\eps.$$
Moreover, we assume that $\bG$ preserves tangential collisions:
$\bG(r,\pm1)=(r,\pm1)$. In other words, $g^i(r,\pm1)=0$ for each $i=1,2$. }

In particular, the singularity set of $T^{-1}_{\bP}$ is the same as that
of untwisted map $T^{-1}_{\bF}$.\\

Let $\cI:\cM\to\cM$ be the reversal  transformation (also called involution),
which is defined by $\cI(x,y,\theta)=(x,y,\pi+\theta)$. Let $\Phi_t:\cM\to\cM$ be a general flow.
The reversed flow of $\Phi$ is defined by $\Phi_t^{-}=\cI\circ \Phi_{-t}\circ \cI$.
Then the flow $\Phi$ is said to be {\it time-reversible}, if $\Phi_t^-=\Phi_t$.
It is well known that the unforced billiard flow is time-reversible. \\

\noindent(\textbf{A4}) (\textbf{Time-reversibility})
\emph{Both forces $\bF$ and $\bG$ are stationary,
and the forced billiard flow $\Phi_{\bP}$ is time-reversible.} \\

Let $\eps>0$, $\tau_\ast\in(0,1)$ and $\cF(Q, \tau_\ast, \eps)$ be the collection
of all forced billiard maps
defined by the dynamics \eqref{flow0} and \eqref{reflectiong} under
the external forces $\mathbf{P}=(\bF,\bG)$ that satisfying the assumptions (\textbf{A1})--(\textbf{A4})
with $\tau_\ast\le\tau_{\min}\le\tau_{\max} \le \tau_\ast^{-1}$.
The following lemma was proved by  Demers and  Zhang in [DZ11]:
\begin{lemma}\label{TmEDC}
Each map $T\in \cF(Q, \tau_\ast, \eps)$ preserves a unique  SRB
measure (the nonequilibrium steady state) $\mu_{T}$ that is mixing, Bernoulli and positive on open
sets in $M$. Let $\cH$ be the collection of all piecewise H\"older
continuous functions on $M$ where the discontinuities occur on the singularity sets $T$. Then
\begin{enumerate}
\item[a.] (Equidistribution) for any $f\in\cH$, $T^n\mu_0(f)\to \mu_T(f)$ at an exponential rate;

\item[b.] (Decay of Correlations) for any $f,g\in\cH$, $\mu_0(f\circ T^n\cdot g)\to \mu_T(f)\cdot\mu_0(g)$
at an exponential rate.

\item[c.] (Central Limit Theorem) for any $f\in\cH$, $f_n=f+f\circ T+\cdots +f\circ T^{n-1}$,
then $\frac{f_n}{\sqrt{n}}\Rightarrow \cN(\mu_T(f),\sigma^2_f)$, where $\sigma^2_f=\sum_{n\in\mathbb{Z}}\cC_{f,f}(T^n)$.
Here the convergence means that the distributions of  $\frac{f_n}{\sqrt n}$
converge to the normal distribution $ \cN(\mu_T(f), \sigma^2_f)$.
\end{enumerate}
Moreover, the flow $\Phi$ preserves a mixing SRB measure
$\hat\mu_{T}$ with similar properties.
\end{lemma}

\subsection{Main results}

In this section we state the main results of this paper, the
properties of forced billiard systems under the assumptions
(\textbf{A1})--(\textbf{A4}). More precisely, let $\eps>0$ be small
enough, $\tau_\ast\in(0,1)$ and $\bP=(\bF,\bG)$ be an $\eps$-small
external force pair such that $T_{\bP}\in \cF(Q, \tau_*, \eps)$,
$\mu_{\bP}$ be the SRB measure on $M$ given in Theorem \ref{TmEDC}.
Let $\tQ$ be the $\mathbb{Z}^2$-periodic table on the plane,
$\tT_{\bP}$ be the forced collision map on the
$\mathbb{Z}^2$-periodic collision space $\tM$, and
\beq\label{displace}
\widetilde\bDelta_{\bP}(\bx)=(\widetilde\Delta_{x,\bP},
\widetilde\Delta_{y,\bP})=\tpi\circ \tT_{\bP}(\bx)-\tpi(\bx)\eeq be
the displacement vector between collisions, where $\tpi$ is the
projection from $\tM$ to the base point in $\tQ$. It is easy to see
that that $\widetilde\bDelta_{\bP}$ is $\mathbb{Z}^2$-periodic,
which induces a displacement function on the collision space $M$,
say $\bDelta_{\bP}$.
Our first theorem describes estimations on the current for the discrete system.\\
\begin{theorem}\label{Thm:2}
(a) The discrete-time  steady state electrical current is well-defined and given
by:
 \beq\label{current}
 \bJ_{\bP}=\lim_{n\to\infty} \tq_n/n=\mu_{\bP}(\bDelta_{\bP}).
 \eeq
(b) The current ${\bJ_{\bP}}$ satisfies $ \bJ_{\bP}=\eps \bsigma+o(\eps)$,
where $\bsigma=(\sigma_{x}, \sigma_{y})$ is uniformly bounded and satisfies
 \beq\label{sigmaa}
\sigma_a=\frac{1}{2}\mu_0(\Delta_{a,\bP}\cdot H)
+\sum_{k=1}^{\infty}\mu_0[(\Delta_{a,\bP}\circ T^k_{0})\cdot H],\quad a\in \{x,y\},
 \eeq
where $H(r,s)=(2-\exp(\int_0^{\tau_{\bF}(r,s)}ph_\theta\, dt)
-\cJ_{\bG}(r_1,s_1))/\eps$ is a  uniformly bounded function, and $\cJ_{\bG}$ is the Jacobian
of the twisting force $\bG$.\\
(c) As $n\to\infty$,
\beq\label{tqnjncon}
\frac{\tq_n- n\bJ_{\bP}}{\sqrt n}\Rightarrow \cN(0, \bD_{\bP}),
\eeq
where $\bD_{\bP}$ is the discrete-time diffusion matrix of
the Lorentz particle, which is given by:
\beq\label{greeneps}
\bD_{\bP}=\sum_{n=-\infty}^{\infty}
\left[\mu_{\bP}(\bDelta_{\bP}\circ T^n_{\bP}\otimes \bDelta_{\bP})
-\mu_{\bP}(\bDelta_{\bP})\otimes \mu_{\bP}(\bDelta_{\bP})\right]
\eeq

\noindent(d) $\bD_{\bP}$ is continuous with respect to the size of
the force pair $\bP$ at $\bP=0$:
 \beq\label{hcDepsdel}
 \bD_{\bP}=\bD_0+o(1).
 \eeq
\end{theorem}
In particular, the generated current $\bJ_{\bP}$ is comparable to
the size of the external force $\bP$, where $\bsigma$ resembles the
electric conductivity of the forced system. Moreover, the drift
effect is sub-linear. One may wonder if we could use the linear approximates
$\int_0^{\tau_{\bF}(r,s)}ph_\theta\, dt$ for the function $H$ in (b).
However, this may destroy the convergence of the series \eqref{sigmaa}
(see Remark \ref{linear} for more discussions).  However,
we do use the linear term for all physical models with a {\it Gaussian
thermostat} (see Proposition \ref{isokinetic},  \ref{forcefield} and Corollary \ref{Const}).  \\

The corresponding
results for the continuous-time  forced system is provided by the following
theorem.
\begin{theorem}\label{Thm:3}
Suppose a particle move in the domain $\tQ$ under the external force $\bP=(\bF,\bG)$.\\
(a) The steady state current generated by $\Phi_{\bP}$ is well-defined and given
by
 \beq\label{currentf}
 \hat\bJ_{\bP}=\lim_{t\to\infty}
 \tq(t)/t=\mu_{\bP}(\bDelta_{\bP})/\mu_{\bP}(\tau_{\bP}).
 \eeq
(b) The current $\hat{\bJ}_{\bP}$ satisfies
 \beq\label{JDEf}
 \hat\bJ_{\bP}=\eps \cdot\frac{\bsigma}{\bar\tau_{\bP}}+o(\eps).
 \eeq
(c) As $t\to\infty$,
\beq\label{tqnjnconf}
\frac{\tq(t)-\hat\bJ_{\bP} t}{\sqrt t}\Rightarrow \cN(0, \hat\bD_{\bP}),
\eeq
where $\hat\bD_{\bP}$ is the continuous-time  diffusion matrix of the Lorentz particle.\\
(d) The diffusion matrix is continuous with respect to the size of
force pair $\bP$ at $\bP=0$:
 \beq\label{hcDepsdelf}
 \hat\bD_{\bP}=\hat\bD_0+o(1)
 \eeq
\end{theorem}
We also give some characterizations of the nonequilibrium steady
state $\mu_{\bP}$ of the forced system $T_{\bP}$:
\begin{theorem}\label{Thm:1}
(1) The measure $\mu_{\bP}$ satisfies the Pesin entropy formula:
\beq\label{pesin}
h_{\mu_{\bP}}(T_{\bP})=\lambda_{\bP}^u,
\eeq
where $\lambda^s_{\bP}<0<\lambda^u_{\bP}$ are the Lyapunov exponents for the measure $\mu_{\bP}$ and $h_{\mu_{\bP}}(T_{\bP})$ is the metric entropy for $(T_{\bP},\mu_{\bP})$;\\
(2) $\mu_{\bP}$ satisfies Young's dimension formula:
\beq\label{HD}
\mathrm{HD}(\mu_{\bP})=h_{\mu_{\bP}}(T_{\bP})\left(\frac{1}{\lambda^u_{\bP}}-\frac{1}{\lambda^s_{\bP}}\right),
\eeq
where $\mathrm{HD}(\mu_{\bP})$ is the Hausdorff dimension of the measure $\mu_{\bP}$.\\
(3) Let $h_0=h_{\mu_0}(T_0)$ be the metric entropy of $T_0$, then
\begin{align}\label{fractal}
\mathrm{HD}(\mu_{\bP})&=2-\eps^2\cdot\frac{\sigma^2_H}{2h_0}+o(\eps^2),\:
\text{ where }\:
\sigma^2_H=\sum_{k=-\infty}^{\infty}\mu_0[H\circ T_0^k\cdot H].
\end{align}
\end{theorem}
It follows from Eq. \eqref{fractal} that
$1<\mathrm{HD}(\mu_{\bP})<2$ for some small external forces $\bP$. So
$\mu_{\bP}$ is singular with respect to the Lebesgue measure and
admits a fractal structure.

Let $\hat\mu_{\bP}$ be the corresponding SRB measure preserved by the forced flow
$\Phi^t_{\bP}$ on $\cM$. The metric entropy of the measure
$\hat\mu_{\bP}$ is given by
$h_{\hat\mu_{\bP}}(\Phi_{\bP}):=h_{\mu_{\bP}}(T_{\bP})/\mu_{\bP}(\tau_{\bP})$
and the fractal dimension
$\mathrm{HD}(\hat\mu_{\bP})=\mathrm{HD}(\mu_{\bP})+1$. Therefore
similar formulas in Theorem \ref{Thm:1} hold for $\hat\mu_{\bP}$.

\subsection{Applications}

Next we provide some example of reversible external forces and
give the generated currents by the forced billiard systems.\\

\noindent\textbf{Example 1. Conservative forces on $Q$.}

We consider a conservative force $\bF=-\nabla U(\bq)=-\langle
U_x,U_y\rangle$, where $U(\bq)$ is a (small) potential function on
$Q$. The induced billiard system $\Phi_{\bF}$ preserves the total
energy $E=\frac{1}{2} p^2+U(\bq)$ of the system and hence satisfies
Assumption ({\bf A1}). We restrict the dynamics to a energy level
$\cM=\{(\bq,\bp):E(\bq,\bp)=1/2\}$. In particular the speed function
$p^2(x,y,\theta)=1-2U(\bq)$ is independent of $\theta$, say
$p=p(x,y)$. It is well known that the billiard flow $\Phi_{\bF}$ is
time-reversible (see also  Lemma \ref{reversible} and Remark
\ref{posiforc}). Moreover, it is easy to see that the generating
vector field $X=\langle p\cos\theta,p\sin\theta,ph\rangle$ is
divergence-free: \beq \text{div
}X=p_x\cos\theta+p_y\sin\theta+(U_x\cos\theta+U_y\sin\theta)/p
=\frac{1}{p}(\dot p-\bv\cdot \bF)=0. \eeq In particular, the flow
$\Phi_{\bF}$ preserves the Lebesgue measure $m$ on the energy level
$\cM$ and $\hat\mu_{\bF}=m$. Since there is no slip after the
collision, the current is indeed zero (see the related discussion in
Remark \ref{flowcurrent}):
 \beq
 \bJ_{\bF}=m(\bp)=\int_Qp(x,y)(\int_{S^1}\bv d\theta) dxdy=0.
 \eeq

\noindent\textbf{Example 2. Conservative twists and soft scatterers.}

Let $\bB$ be a small scatterer on $Q$, $U$ a
potential function on $\mathbb{T}^2$ such that $U=0$ on $Q\backslash
\bB$ and $U>0$ on $\bB$. We
consider the energy surface $E(\bq,\bp)=\frac{1}{2}\|\bp\|^2+U(\bq)=1$.
So a running billiard may climb up the scatterer, and regain its
full kinetic energy whenever it exits that scatterer. If we view
what happened on the scatterer as a black box, the reduced dynamics
is close to the unforced system during its free flight. The only
difference is that the exit location and direction are different
from the elastic collision (corresponding to $U=+\infty$ on $\bB$).
Therefore, the effect is equivalent to applying a twist force $\bG$
right after the collision. Clearly the perturbed flow $\Phi$
is time-reversible. Moreover, by embedding $\Phi$ into the
real Hamiltonian flow on the ambient torus, we see that
$\Phi_{\bG}$ preserves the canonical space $\cM_0$ and the
Liouville measure on $\cM_0$. In particular its current
$\bJ_{\bG}=0$.

To get an intuition of the vanishing currents, we consider
a special case explicitly.
\begin{proposition}
Suppose all the scatterers $\bB_i$ are all round disks
centered at $\bq_i=(x_i,y_i)$ of radius $\delta_i$,
the potential is given by $U(\bq)=\eps^{-1}\cdot(\delta_i-\|\bq-\bq_i\|)$ on $\bB_i$,
and $U=0$ outside the scatterers. Then the forced flow has vanishing current.
\end{proposition}
\begin{proof}
In the case it is easy to see that
the outgoing direction is the same as the hard scatterer case,
the slip does not depend on where the collision happens:
$\bG(r,s)=(r+g_i(s),s)$. So the Jacobian $\cJ_{\bG}\equiv 1$.
Therefore the current must be zero.
\end{proof}

\noindent\textbf{Example 3. Isokinetic forces.}

Next we consider forces $\bF$ that are always perpendicular to
the momentum. In this case, the forced billiard flow preserves the
kinetic energy. Without loss of generality, we assume $\|\bp\|=1$
(as in the classical billiards), and reformulate the force as
$$\bF=F\bv^{\perp}=F(x,y,\theta)(-\sin\theta, \cos\theta),$$
where $F=F(x,y,\theta)$ is a scale function (may be negative) with
$\|F\|_{C^1}<\eps$. By Eq. (\ref{formh}), the geometric curvature
$h(x,y,\theta)$ satisfies
 \beq\label{isoforce}
h(x,y,\theta)=F(x,y,\theta)(\sin^2\theta+\cos^2\theta)/p^2=F(x,y,\theta).
 \eeq
Not all
isokinetic forced systems are time-reversible. For example, the system of an
electron moving under a constant magnetic field perpendicular to $Q$
is not time-reversible.

In the following we assume that $F$ satisfies
$F(x,y,\theta+\pi)=-F(x,y,\theta)$. Then the forced billiard flow
$\Phi_{\bF}$ is time-reversible (see Lemma \ref{reversible}). A
special feature of these isokinetic forces is that we can use the
linear term of $H$ to estimate the generated current:
\begin{proposition}\label{isokinetic}
The discrete-time  steady state current under a general isokinetic force
$\bF$ is given by
$\bJ_{\bF}=\eps\bsigma+o(\eps)$, where
$\bsigma=(\sigma_x,\sigma_y)$ is given by
\begin{equation*}
\sigma_a=-\frac{1}{2}\mu_0[\Delta_{a,\bF}\cdot
\int_0^{\tau_{\bF}( \bx)}F_{\theta}\, dt/\eps]
-\sum_{k=1}^{\infty}\mu_0[(\Delta_{a,\bF}\circ T^k_{0})\cdot
\int_0^{\tau_{\bF}( \bx)}F_{\theta}\, dt/\eps], \quad a\in\{x,y\}.
\end{equation*}
\end{proposition}

\noindent{\bf Example 4. }\textbf{Electric field with thermostat.}
We consider an electric field $\bE(\bq)=(\eps e_1(\bq),\eps e_2(\bq))$ on $\bq\in Q$.
It may generate a net velocity in the
force direction and keep accelerating the electron. We can modify the system by adding a constraining
force to maintain the system at a constant temperature (a compact
level set, say $\cE_1=\{\|\bp\|\equiv 1\}$), and to preserve a steady state on that
level. More precisely, the system of the forced equations on $\cE_1$ is given by:
 $\dot
\bq=\bp$, $\dot\bp=\bE-\alpha \bp$, where $\alpha=\bE\cdot \bp$
is a thermostat. For such systems, we have the following results.

\begin{proposition}\label{forcefield}
Let $\bE(\bq)=(\eps e_1(\bq),\eps e_2(\bq))$ be  an electric field on $\bq\in Q$.
The discrete-time  steady state electrical current under $\bE$ with
thermostat is given by
$\bJ_{\bE}=\eps\bsigma+o(\eps)$, where
$\bsigma=(\sigma_x,\sigma_y)$ is given by
$$\sigma_a=\frac{1}{2}\mu_0(\Delta_{a,\bE}\cdot \int_0^{\tau_{\bE}}(e_1,e_2)\cdot \bp\,  dt)+
\sum_{k=1}^{\infty}\mu_0[(\Delta_{a,\bE}\circ T^k_{0})\cdot \int_0^{\tau_{\bE}}(e_1,e_2)\cdot \bp\,  dt].$$
\end{proposition}

Now we consider a constant electric
field $\bE=(\eps e^0_1,\eps e^0_2)$ on the table $Q$. In this case it is easy to see
that $\int_0^{\tau_{\bE}(r,s)}(e^0_1,e^0_2)\cdot \bp  dt=(e^0_1,e^0_2)\cdot \bDelta_{\bE}(r,s)$.
So we have
\begin{corollary}\label{Const}
Let $\bE(\bq)=(\eps e^0_1,\eps e^0_2)$ be a constant electric field on $\bq\in Q$.
The discrete-time  steady state electrical current under $\bE$ with
thermostat is given by
$ \bJ_{\bE}=\eps \bsigma+o(\eps)$,
where $\bsigma=(\sigma_{x}, \sigma_{y})$ is given by
 \begin{align*}
\sigma_a=&\frac{1}{2}\mu_0(\Delta_{a,\bE}\cdot H)
+\sum_{k=1}^{\infty}\mu_0[(\Delta_{a,\bE}\circ T^k_{0}) \cdot H],\quad a\in \{x,y\},
 \end{align*}
 where $H(r,s)= (e^0_1,e^0_2)\cdot\bDelta_{\bE}(r,s)$.
 \end{corollary}
This is  the current formula obtained in \cite{CELS1}.

\section{Preliminary properties of $T_{\bF}$ under the force $\bF$}

We divide our study of the forced system $T_{\bP}$ into two steps
according to the nature of the force pair $\bP$: the step of a
elastic reflection under the force $\bF$, and the twisting by $\bG$
right after the collision. In this section we consider the pre-twist
step, that is, the effect of the force $\bF$ between one collision.
First we state the time-reversibility of the forced system:

\begin{lemma}\label{reversible}
Let $\bF$ be an external force on the table $Q$,
$p$  and $h$ be the speed and the curvature functions of the forced billiard flow $\Phi_{\bF}$.
Then $\Phi_{\bF}$ is time-reversible if and only if the following conditions hold for any
$(x,y,\theta)\in \cM$:
\begin{subequations}\label{reverse}
 \begin{align}
    p(x,y,\pi+\theta)&=p(x,y,\theta), \label{reverse-a}\\
    h(x,y,\pi+\theta)&=-h(x,y,\theta). \label{reverse-b}
  \end{align}
\end{subequations}
\end{lemma}
The proof is straightforward and omitted here. See also \cite[P.
209--210]{Ch01} for detailed discussions.

\begin{remark}
There is a canonical involution $I$ on the post-collision space $M$, which is given by
$I:M\to M$, $(r,s)\mapsto(r,-s)$.  Let $T^-_{\bF}$ be the Poincar\'e map
of the reversed flow $\Phi_{\bF}^-$. It is easy to see that $T^-_{\bF}=I\circ T^{-1}_{\bF}\circ I$.
So time-reversibility of the forced flow
$\Phi_{\bF}^-=\Phi_{\bF}$ implies the time-reversibility of the forced map:
$T_{\bF}^-=T_{\bF}$.
\end{remark}

\begin{remark}\label{posiforc}
A special case is that the force $\bF=\bF(\bq)$ depends only on the
position $\bq$. For example, for a particle moving in the
gravitivity field and an electron moving in an electric field, the
forces do not depend on the velocity $\bp$. According to Eq.
\eqref{formh}, we see that \eqref{reverse-b} follows from
\eqref{reverse-a}. So these forced systems are
time-reversible if and only if $p(x,y,\pi+\theta)=p(x,y,\theta)$ for
any $(x,y,\theta)\in \cM$.
\end{remark}

\medskip

\begin{lemma}\label{TFJ}
Let $m$ be the Lebesgue measure on $M$,  $\bx=(r,s)$ and
$T_{\bF}\bx=(r_1,s_1)$. Then the Jacobian of $DT_{\bF}$ with respect
to the Lebesgue measure $\mu_0$ is given by \beq\label{cJeps} \det
D_{\bx}T_{\bF}=\exp(\int_0^{\tau_{\bF}( \bx)}ph_\theta dt). \eeq
\end{lemma}
\begin{proof}
Let $X(x,y,\theta)=\langle p\cos\theta,p\sin\theta,  p h\rangle$ be
the vector field on $\cM$ that generates the flow $\Phi^t_{\bF}$.
Consider the Lebesgue measure $dm=dx\, dy\, d\theta$ on $\cM$. Note
that $m$ is not necessarily invariant under the forced flow
$\Phi_{\bF}$, and its rate of change is given by the divergence of
the generating vector field $X$:
\begin{equation*}
\text{div }X(x,y,\theta)=p_x\cos\theta+p_y\sin\theta+p_{\theta}h+ ph_{\theta}
=(p_x\dot x+p_y\dot y+p_{\theta}\dot \theta)/p+ ph_{\theta}=\frac{d\ln p}{dt}+ ph_{\theta}.
\end{equation*}

Note that the force billiard flow can also be represented as a suspension of the
forced map $(M,T_{\bF},\mu_0)$ with respect to the roof function
$\tau_{\bF}$. In particular the volume $dm=dx\, dy\, d\theta$ on
$\cM$ has the suspension form: $dm=\cst\cdot p\cdot d\mu_0\, dt$.
Clearly the $t$-direction is  flow-invariant. So we have
 \beq\label{defcJ} \frac{dT^{-1}_{\bF}\mu}{d\mu}(\bx) = \frac{p(
\bx)}{p(T_{\bF} \bx)}\exp(\int_0^{\tau_{\bF}( \bx)} \text{div
}X(\Phi^t_{\bF} \bx)\, dt) =\exp(\int_0^{\tau_{\bF}( \bx)}ph_\theta
dt).
 \eeq
This finishes the proof of the lemma.
\end{proof}

\section{Properties of $T_{\bP}=T_{(\bF,\bG)}$}

We consider the twisting step of a force pair $\bP=(\bF,\bG)$
satisfying the assumptions (\textbf{A1})--(\textbf{A4}) with
$\|\bF\|_{\cC^1}<\eps$ and $\|\bG-\mathrm{Id}_M\|_{\cC^1}< \eps$.
Our main approach is to
compare the combined effect $T_{\bP}\bx={\bG}\circ T_{\bF}\bx$
with the pre-twist map $T_{\bF}$.

In Lemma \ref{reversible} we gave a characterization for the forced
flow $\Phi_{\bF}$ to be time-reversible. To ensure that $\Phi_{\bP}$ is
time-reversible, it suffices to know if the twisting process is also
time-reversible. More precisely, let $S=\begin{pmatrix} 1 & 0 \\ 0 &
-1
\end{pmatrix}$, and $G=\begin{pmatrix}
{\bf 0} & S \\ S & {\bf 0}
\end{pmatrix}$, where ${\bf 0}$ is the $2\times 2$ zero matrix. Then we have the following result.
\begin{lemma}\label{reversetwist}
Assume the condition \eqref{reverse} hold. Then $\Phi_{\bP}$ is
time-reversible if and only if the graph of ${\bG}$ is
$G$-invariant.
\end{lemma}
\begin{proof}
Let $T_{\bP}\bx={\bG}\circ T_{\bF}\bx$ be decomposition of the forced collision map,
$\Phi_{\bP}^-$ be the time-reversal flow of $\Phi_{\bP}$, and
$T_{\bP}^-$ be the induced Poincare map of $\Phi_{\bP}^-$.
Let $\cI:\cM\to\cM$ be the involution on $\cM$ and $I:M\to M$ be the
induced involution on $M$.

Under the assumption \eqref{reverse}, we have that $T^{-}_{\bF}=T_{\bF}$ and
$$T_{\bP}^-\bx=I\circ {\bG}^{-1}\circ T^{-1}_{\bF}\circ I(\bx)
=I\circ {\bG}^{-1}\circ I\circ I\circ T^{-1}_{\bF}\circ I(\bx)
={\bG}^-\circ T_{\bF}^-(\bx)={\bG}^-(T_{\bF}\bx),$$
where ${\bG}^-=I\circ {\bG}^{-1}\circ I$.
In order to have $T_{\bP}^-=T_{\bP}$, a sufficient condition is ${\bG}^-={\bG}$.

Let $(\bar r,\bar s)={\bG}(r,s)$. Then ${\bG}^-(\bar r,-\bar
s)=I\circ {\bG}^{-1}(\bar r,\bar s)=I(r,s)=(r,-s)$. So
${\bG}^-={\bG}$ is equivalent to ${\bG}(\bar r,-\bar
s)=(r,-s)$. Combining with the assumption $(\bar r,\bar
s)={\bG}(r,s)$, we see that ${\bG}^-={\bG}$ is equivalent to
the graph of ${\bG}$ being $G$-invariant.
\end{proof}

\begin{remark}
In above lemma we take the Poincar\'e map $T^-_{\bP}$ of the time-reversed
flow, which can be viewed as the {\it physically} time-reversal of
the map $T_{\bP}$. We can also define a {\it formally} time-reversed map
of the forced map $T_{\bP}$ by $\widehat T^-_{\bP}:=I\circ T^{-1}_{\bP}\circ I$.
Generally speaking, this formally reversed map $\widehat T^-_{\bP}$ may
not coincide with the Poincar\'e map $T^-_{\bP}$ of the time-reversed flow
$\Phi^-_{\bP}$ (after we apply a twist force right
after the elastic collision). In this paper we always use the
physically time-reversal definition.
\end{remark}

\medskip

By Assumption ({\bf A3}), the twist force $\bG$ preserves tangential
collisions. Therefore, the discontinuity set of $T_{\bP}$ is the
same as that of $T_{\bF}$, which  comprises the preimage of
$\cS_0:=\{s=\pm \pi/2\}$.   Similarly, the singularity sets of
$T_{\bP}^{-1}$ and $T_{\bF}^{-1}$ are the same due to (\textbf{A3}).
But the singular sets for higher iterates are no longer the same.
Let $\cS_{\pm n}^{\bP}= \cup_{i=0}^n T_{\bP}^{\mp i}\cS_{0,H}$ with
$n\in \mathbb{N}$. Then $T_{\bP}^{\pm n}$ is smooth on all the cells
of $M\setminus \cS_{\pm n}^{\bP}$.

For any phase point $ \bx =(r,s)\in M$, let $T_{\bF}  \bx =(r_1,
s_1)$ and $T_{\bP} \bx =(\bar r_1,\bar s_1)$. According to the
discussion between {\bf (A2)} and {\bf (A3)}, we express the twist
force $\bG$ in local coordinates via two smooth functions $g^1$ and
$g^2$ such that
 \beq \label{rphi}
 (\bar r_1,\bar s_1)=\bG( r_1, s_1)=( r_1, s_1)+(g^1( r_1, s_1),g^2( r_1, s_1)).
  \eeq
Note that $g^i$ is a $C^2$ function whose $C^1$ norm is uniformly
bounded from above by $c\cdot\eps$, for some uniform constant $c>0$.
Moreover $g^i(r,\pm 1)=0$, $i=1,2$.

According to  (\ref{rphi}), the differential of the twisting force
$\bG$ satisfies
 \beq \label{drphi}
 \begin{cases}
 d\bar r_1=\left(1+g^1_{1}( r_1, s_1)\right) dr_1+g_{2}^1( r_1, s_1) ds_1,\\
 d\bar s_1=g^2_1( r_1, s_1) dr_1+\left(1+g_2^2( r_1, s_1)\right)ds_1,
 \end{cases}
 \eeq
where $g^i_1=\partial g^i/\partial r$ and $g^i_2=\partial g^i/\partial s$.
So the differential of the map $T_{\bP}$ is given by
\begin{align}
\label{DTepsg}
D_{\bx}T_{\bP}&=D_{T_\bF\bx}\bG\circ D_{\bx}T_{\bF}.
  \end{align}
Note that $T_{\bP}$ may not be a $C^1$-perturbation of $T_{\bF}$,
since $T_{\bF}$ is unbounded around the boundary of $M$.
However, it follows from Eq. \eqref{DTepsg} that
 \beq
\det D_{\bx}T_{\bP}=\cJ_{\bG}(r_1,s_1)\cdot \det D_{\bx}T_{\bF}
    =\left(1+\hat g(r_1,s_1)\right)\cdot\det D_{\bx}T_{\bF},
 \eeq
where
\begin{align}\label{hatg}
 \hat g(r_1,s_1)&=\cJ_{\bG}(r_1,s_1)-1\nonumber\\
 &=g^1_1( r_1, s_1)+g^2_2( r_1, s_1)
 +g^1_1( r_1, s_1)g^2_2( r_1, s_1)-g^1_2( r_1, s_1)g^2_1( r_1, s_1).
\end{align}
Note that $\cJ_{\bG}$ is a $C^1$ function with $\hat g(r_1,s_1)=\cO(\eps)$. So we have that
\begin{align}
\label{differential} \det D_{\bx}T_{\bP}= \left(1+\hat
g(r_1,s_1)\right)\cdot \exp(\int_0^{\tau_{\bF}( \bx)}ph_\theta dt)
=1 + \cO(\eps).
\end{align}

Once again, let $\mu_0=\cst\, dr\, ds$ be the normalized Lebesgue measure on
$M$, $\cJ_{\bP}(\bx):=d T_{\bP}^{-1}\mu_0/d\mu_0(\bx)$ be the
density function defined by $T_{\bP}$. Note that
$\cJ_{\bP}(\bx)=\det D_{\bx}T_{\bP}$. So we have:
\begin{lemma}\label{cJP}
Let $\bP=(\bF,\bG)$ be an $\eps$-small force pair, and $T_{\bP}$
the forced collision map. Let $\bx=(r,s)\in M$,
$T_{\bF}\bx=(r_1,s_1)$ and $T_{\bP}\bx=(\bar r_1, \bar s_1)$. Then
the Jacobian $\cJ_{\bP}(\bx)$ satisfies
 \beq\label{appcJ}
 1-\cJ_{\bP}(\bx)=\eps H(\bx)+\eps^2 R_{\bP},
 \eeq
where
\begin{align*}
H(r,s)&=\frac{1}{\eps}\left(2-\exp(\int_0^{\tau_{\bF}(r,s)}ph_\theta dt)-\cJ_{\bG}(T_{\bF}\bx)\right)\\
&=\frac{1}{\eps}\left(1-\exp(\int_0^{\tau_{\bF}(r,s)}ph_\theta dt)-\hat
g(r_1,s_1)\right).
\end{align*}

Moreover, $\mu_0(H)=0$, and both $H$ and $R_{\bP}$ are uniformly bounded and
$C^1$ on each component of $M\backslash \cS^T_1$.
\end{lemma}

\begin{proof}
According to Chain Rule and Lemma \ref{TFJ}, we have
\begin{align}\label{differentialP}
\cJ_{\bP}(\bx)&=\frac{dT^{-1}_{\bP}\mu_0}{d\mu_0}(\bx)
=\frac{dT^{-1}_{\bP}\mu_0}{dT^{-1}_{\bF}\mu_0}(\bx)
\cdot\frac{dT^{-1}_{\bF}\mu_0}{d\mu_0}(\bx)\nonumber\\
&=\cJ_{\bG}(T_{\bF}\bx)\cdot\det D_{\bx}T_{\bF}
=\left(1+\hat g(r_1,s_1)\right)\cdot\exp(\int_0^{\tau_{\bF}( \bx)}ph_\theta dt).
\end{align}

Let $H(\bx)=(1-\exp(\int_0^{\tau_{\bF}(r,s)}ph_\theta dt)-\hat g(r_1,s_1))/\eps$.
Then we get the following expansion for $\cJ_{\bP}$ at $\bP=0$:
 \beq\label{appcJ1} \cJ_{\bP}(\bx)=\cJ_0(\bx)-\eps H(\bx)-\eps^2
R_{\bP},
 \eeq
where $R_{\bP}$ is the residual term (up to a factor $\epsilon^2$).

Now it is easy to see that $\mu_0(H)=0$. Firstly, $1+\mu_0(\hat
g)=\mu_0(1+\hat g)=\mu_0(\cJ_{\bG})=\mu_0(M)=1$. So $\mu_0(\hat
g)=0$. Secondly, we note that
$\mu_0(e^{\int_0^{\tau_{\bF}(r,s)}ph_\theta
dt})=\mu_0(DT_{\bF})=\mu_0(T_{\bF}M)=\mu_0(M)=1$. Therefore, we have
$\mu_0(H)=0$.

By assumptions (\textbf{{A1}})--(\textbf{{A3}}), both $H$ and
$R_{\bP}$ are uniformly bounded on $M$, and are $C^1$ continuous
functions on each component of $M\backslash \cS^{\bP}_1$. This completes
the proof.
\end{proof}

Let $\mu_{\bP}$ be the SRB measure of $T_{\bP}$ on $M$ given by
Theorem \ref{TmEDC}. This measure represents the natural
non-equilibrium steady state (NESS) for the system (see
\cite{Ru99}), which might be singular with respect to the Liouville
measure on $M$. There are several physically interested quantities
associated to the NESS $\mu_{\bP}$. For example, the current of the
billiard flow on the $\mathbb{Z}^2$-periodic table $\tQ$ is given by
$$\bJ_{\bP}=\mu_{\bP}(\bDelta_{\bP})=\mu^o_{\bP}(\widetilde\bDelta_{\bP}),$$
where $\bDelta_{\bP}$ is  the induced displacement vector function
on $M$, $\mu_{\bP}^o$ is a copy of $\mu_{\bP}$ on a fundamental
domain (say $\tM_o$) of $\tM$. This current was derived in
\cite{CELS1,CELS2} when the billiard is moving in an electric field
with a Gaussian thermostat. The study of the current as a function
of a general electric field was carried out in \cite{Ch01,Ch08}. In
the other direction, these results were generalized in \cite{Z11} to
systems where the collision rule is perturbed. Here we study the
current for billiards under more general external forces. Firstly we
will prove the linear response property for general observable with
respect to the perturbed billiard system.
\begin{lemma}
Let $\bP=(\bF,\bG)$ be an $\eps$-small force pair, $T_{\bP}$
be the induced billiard map and $\mu_{\bP}$ be the SRB measure of $T_{\bP}$
on $M$, $\cH$ be the set of piecewise Holder continuous functions on $M$
whose discontinuities occur at the singularities $S_1^{\bP}$. Then for any $f\in  \cH$, we have
\beq\label{kawasaki}
\mu_{\bP}(f)=\mu_0(f)+\eps \cdot \sum_{k=1}^{\infty}\mu_0[(f\circ T^k_{0})\cdot H]+o(\eps)\eeq
\end{lemma}
\begin{proof}
Let $f\in  \cH$. Then for all $n\ge1$, we have the following identity:
\begin{align*}
T_{\bP}^n\mu_0(f)-\mu_0(f)&=\sum_{k=1}^{n}(T_{\bP}^{k}\mu_0(f)-T_{\bP}^{k-1}\mu_0(f))
=\sum_{k=1}^{n}\mu_0[(f\circ T^k_{\bP})(1-\cJ_{\bP})]\\
&=\sum_{k=1}^{n}\mu_0[(f\circ T^k_{\bP})(\eps H(\bx)+\eps^2 R_{\bP})].
\end{align*}
Note that $T_{\bP}^n\mu_0(f)\to\mu_{\bP}(f)$
exponentially (by Theorem \ref{TmEDC}-a). Passing $n\to\infty$, we
get that
\begin{align}\label{Kawasa}
\mu_{\bP}(f)&=\lim_{n\to\infty}T_{\bP}^n\mu_0(f)
=\mu_0(f)+\sum_{k=1}^{\infty}\mu_0[(f\circ T^k_{\bP})(\eps H(\bx)+\eps^2
R_{\bP})].
\end{align}
It follows from the fact that $\cJ_{\bP}$ is the density function of
a probability measure, that $\mu_0(1-\cJ_{\bP})=0$. Combining with
the fact that $\mu_0(H)=0$, we get that $\mu_0(R_{\bP})=0$, too. In
addition $H$ and $R_{\bP}$ are piecewise $C^2$ functions whose
discontinuities occur only at the singularities of $T_{\bP}$. Thus $H$ and
$R_{\bP}$ belong to $\cH$. Now (\ref{Kawasa}) implies that
\begin{align}\label{similar}
&\sum_{k=1}^{\infty}\mu_0[(f \circ T^k_{\bP})(\eps H+\eps^2 R_{\bP})]
=\eps\sum_{k=1}^{\infty}\mu_0[(f\circ T^k_{\bP})  H] +\eps^2\sum_{k=1}^{\infty}\mu_0[(f \circ T^k_{\bP}) R_{\bP}] \nonumber\\
=&\eps \sum_{k=1}^{\infty}\mu_0[(f \circ T^k_{0})H]
+\eps \sum_{k=1}^{\infty}\mu_0[(f \circ T^k_{\bP}-f\circ T^k_0)H]
+\eps^2\sum_{k=1}^{\infty}\mu_0[(f \circ T^k_{\bP}) R_{\bP}]\nonumber\\
=&\eps \sum_{k=1}^{\infty}\mu_0[(f \circ T^k_{0})H]+o(\eps),\end{align}
where we have used the Lebesgue Dominated Convergence Theorem in the
last step, since both $\sum_{k=1}^{\infty}\mu_0[(f \circ
T^k_{\bP})H]$ and $\sum_{k=1}^{\infty}\mu_0[(f \circ T^k_{0})H]$
converge exponentially fast (by Theorem \ref{TmEDC}-b) and
$T^k_{\bP}\to T^k_{0}$ as $\bP\to 0$. Therefore the series are dominated by
$o(1)$. This completes the proof.
\end{proof}

Let $(x,y)$ be the coordinates of the position of the particle
$\tq\in \tQ$, and
 \beq \widetilde\bDelta_{\bP}:=(\Delta_{x,\bP},
\Delta_{y,\bP})=\tpi\circ \tT_{\bP}-\tpi,
 \eeq
be the displacement vector map defined on $\tM$. As pointed out
before, $\widetilde\bDelta_{\bP}$ is $\mathbb{Z}^2$-periodic and
hence induces a well-defined, $\mathbb{R}^2$-valued function
$\bDelta_{\bP}$ on $M$.
Using the fact that $\mu_0(\bDelta_0)=0$ and $\mu_0(\tau_0)=\bar \tau$, we get the following results.
\begin{lemma}\label{deltataueps}
Both functions $\bDelta_{\bP} $ and $\tau_{\bP}$ belong to $ \cH$ with Holder exponent $\frac{1}{2}$ and uniformly bounded. Moreover,
$$ \mu_0(\Delta_{\bP})=\cO({\eps})\,\,\,\,\,\text{ and }\,\,\,\,\,\mu_0(\tau_{\bP})=\bar\tau+\cO({\eps})$$
\end{lemma}
\noindent The proof is omitted here, since H\"older continuity
follows closely from \cite[Lemma 8.2]{DC} and the estimations are
almost identical with the proof of Lemma 8.1 in \cite{Z11}.\\

Next we calculate the current for the moving billiards under small
force depending on $\eps$. We only consider the case when the new
system $T_{\bP}$ is time-reversible.
\begin{theorem}\label{currentP}
Suppose the system is time-reversible under the external forces
$\bP=(\bF,\bG)$. Then
\begin{align}\label{bJrev}
\bJ_{\bP}:=\mu_{\bP}(\bDelta_{\bP})=\eps \bsigma+o(\eps),
\end{align}
where $\bsigma=(\sigma_x,\sigma_y)$ is given by
\beq\label{sigma}
\sigma_a=\frac{1}{2}\mu_0(\Delta_{a,\bP}\cdot H) +\sum_{k=1}^{\infty}\mu_0[(\Delta_{a,\bP}\circ T^k_{0})\cdot H],\quad a\in\{x,y\}.
\eeq
Moreover, the current for the flow satisfies $\hat\bJ_{\bP}=\eps \bsigma/\bar\tau_{\bP}+o(\eps)$,
where $\bar\tau=\mu_{0}(\tau_{0})$.
\end{theorem}
\begin{proof}
We first use the invariance of $\mu_{\bP}$ and the Kawasaki formula (\ref{Kawasa}) to write
$\bJ_{\bP}$ as
\begin{align}\label{nuepsD}
&\bJ_{\bP}=\mu_{\bP}(\Delta_{\bP} )=\frac{1}{2}(\mu_{\bP}(\Delta_{\bP} )
+\mu_{\bP}(\Delta_{\bP} \circ T^{-1}_{\bP}))\\
=&\frac{1}{2}\left(\mu_0(\Delta_{\bP} ) +\mu_0(\Delta_{\bP} \circ T^{-1}_{\bP})+\mu_0[\Delta_{\bP}(1-\cJ_{\bP})]\right)+\sum_{k=1}^{\infty}\mu_0[(\Delta_{\bP} \circ T^k_{\bP})(1-\cJ_{\bP})].\nonumber
\end{align}
Recall that $\tM$ is the collision space of the
$\mathbb{Z}^2$-periodic table $\tQ\subset\mathbb{R}^2$. Let $\tM_o$
be a fundamental domain of $\tM$, say
$(\partial\tQ\cap[0,1]^2)\times[-1,1]$, and $\mu_0^o=\cst\, dr\, ds$
be a copy of $\mu_0$ on $\tM_o$. So for any function $f$ on $M$,
$\mu_0(f)=\mu_0^o(\tf)$, where $\tf$ is a $\mathbb{Z}^2$-periodic
lift of $f$ on $\tM$.

Since $\mu_0^o$ is evenly distributed on $\tM_o$, it is invariant
under the involution $I$. In particular, we have
 \beq
 \mu_0^o(\tpi\circ
\tT_{\bP})=\mu_0^o(\tpi\circ \tT^-_{\bP}) =\mu_0^o(\tpi\circ
I\circ\tT^{-1}_{\bP}\circ I)=\mu_0^o(\tpi\circ \tT_{\bP}^{-1}).
 \eeq
This implies that
\begin{align*}
&\mu_0(\bDelta_{\bP})+\mu_0(\bDelta_{\bP}\circ T^{-1}_{\bP})
=\mu_0^o(\widetilde\bDelta_{\bP})+\mu_0^o(\widetilde\bDelta_{\bP}\circ \tT^{-1}_{\bP})\\
=&\mu_0^o(\tpi\circ \tT_{\bP}-\tpi)+\mu_0^o(\tpi-\tpi\circ \tT^{-1}_{\bP})
=\mu_0^o(\tpi\circ \tT_{\bP})-\mu_0^o(\tpi\circ \tT_{\bP}^{-1})=0.
\end{align*}
See \cite{CELS1,DC} for related discussions.
Then according to (\ref{appcJ}), the two components of the current are given by
$$\mu_{\bP}(\Delta_{a,\bP})
=\frac{1}{2}\mu_0[\Delta_{a,\bP}(\eps H+\eps^2\cR)]
+\sum_{k=1}^{\infty}\mu_0[(\Delta_{a,\bP}\circ T^k_{\bP})\cdot
(\eps H+\eps^2\cR)]=\eps\sigma_a+o(\eps),$$
 where
\begin{align}\label{Kawasa1}
\sigma_a=\frac{1}{2}\mu_0(\Delta_{a,\bP}\cdot H)+
\sum_{k=1}^{\infty}\mu_0[(\Delta_{a,\bP}\circ T^k_{0})\cdot H].
\end{align}
Here we use a similar argument in the last step as in
\eqref{similar}, since $T_{\bP}\to T_{0}$ as $\bP \to 0$.

Denote $\bsigma=(\sigma_x, \sigma_y)$.
So we have shown that $\bJ_{\bP}=\eps \bsigma +o(\eps)$. The current of the forced
flow $\Phi_{\bP}$ is given by $\hat\bJ_{\bP}=\mu_{\bP}(\bDelta_{\bP})/\bar\tau_{\bP}
=\eps \bsigma/\bar\tau_{\bP}+o(\eps)$,  where $\bar\tau_{\bP}=\mu_{\bP}(\tau_{\bP})$.
This completes the proof.
\end{proof}

\begin{remark}\label{linear}
It is worth to point out that $\mu_0(H)=0$ is crucial to define
$\sigma_a$, $a\in \{x,y\}$. Otherwise the series in \eqref{Kawasa1}
used to define $\sigma_a$ don't even converge if we take
the linear term $-\int_0^{\tau_{\bF}(r,s)}ph_\theta dt$ . This is the main
reason why we keep all nonlinear terms in Lemma \ref{cJP}. However,
we can use the linear term $\int_0^{\tau_{\bF}(r,s)}ph_\theta dt$ in
all physical models with a {\it Gaussian
thermostat} (see \cite{MHB,CELS1,CELS2,Ch01}). This is exactly the content
in Proposition \ref{isokinetic},  \ref{forcefield} and Corollary \ref{Const}.
\end{remark}

\begin{proof}[Proof of Proposition \ref{isokinetic}]
When restricted to the level set $\cE_1=\{\|\bp\|\equiv 1\}$,
we have $h=F$.
Then we only need to prove that the series \eqref{Kawasa1} do converge
with the linear choice $H=\int_0^{\tau_{\bF}( \bx)}F_{\theta}\, dt$.
By Theorem \ref{TmEDC}-b and Lemma \ref{deltataueps},
it is sufficient to show that $\mu_0(\int_0^{\tau_{\bF}( \bx)}F_{\theta}\, dt)=0$.
Then according to the fact that $dm=\cst \cdot p_0\cdot   d\mu_0\, dt=dx\, dy\, d\theta$, we get
\begin{align}\label{vanish}
\mu_0(\int_0^{\tau_{\bE}( \bx)}F_\theta\, dt)
=&\cst\cdot\int_M\int_0^{\tau_{\bE}( \bx)}F_\theta\, dt d\mu_0
=\int_{\cM}F_\theta dm\nonumber\\
=& \int_{Q}\left(\int_{S^1}F_\theta\,d\theta\right)dx\, dy
=\int_{Q} \left(F|_{\partial S^1}\right)dx\, dy =0.
\end{align}
This finishes the proof of Proposition \ref{isokinetic}.
\end{proof}

\begin{proof}[Proof of Proposition \ref{forcefield}]
Let $\bE=(\eps e_1(\bq), \eps e_2(\bq))$ be the electric field on $Q$,
$\dot \bq=\bp$, $\dot \bp=\bE-\alpha\bp$ be the thermostatted system.
When restricted to the level set $\cE_1=\{\|\bp\|\equiv 1\}$,
the function $h=-e_1\sin\theta+e_2\cos\theta$, and
$h_\theta=-e_1\cos\theta-e_2\sin\theta=-(e_1,e_2)\cdot \bp$.
So $H=-\int_0^{\tau_{\bF}( \bx)}h_{\theta}\, dt
=(e_1,e_2)\cdot\int_0^{\tau_{\bE}( \bx)} \bp\, dt=(e_1,e_2)\cdot \bDelta_{\bE}(\bx)$.
This finishes the proof of Proposition \ref{forcefield}.
\end{proof}

\begin{remark}\label{flowcurrent}
There is another classical representation of the current of the
billiard flow $\hat\bJ=\hat\mu(\bp)$, which may not hold for general
forces. More precisely, let $\hat\mu_{\bP}$ be the SRB
measure of the flow $\Phi_{\bP}$, which can also be obtained as the
suspension of $\mu_{\bP}$ on $\cM$. Then
$\hat\bJ_{\bP}=\hat\mu_{\bP}(\bp)$ holds if and only if
$\displaystyle
\bDelta_{\bP}(\bx)=\bDelta_{\bF}(\bx):=\int_0^{\tau_{\bF}(\bx)}\bp
dt$, that is, there is no slip on $\pQ$ after each collision. For a
general twist force with slip, we would have an additional term
$\displaystyle
\bDelta_{\bG}(T_{\bF}\bx)=\bDelta_{\bP}(\bx)-\bDelta_{\bF}(\bx)$. So
we have $\mu_{\bP}(\bDelta_{\bP})=\int_{M_o}
(\tpi(\tT_{\bP}\bx)-\tpi(\tT_{\bF}\bx))d\mu_{\bP}^o+\mu_{\bP}(\bDelta_{\bF})$.
Since $\mu_{\bP}$ is $T_{\bP}$-invariant, we get an modified formula
of the current for the forced flow
\begin{equation}
\hat\bJ_{\bP}=\hat\mu_{\bP}(\bp)
+\frac{1}{\mu_{\bP}(\tau_{\bP})}\int_{M_o} (\tpi(\bx)-\tpi({\bG}^{-1}\bx))d\mu^o_{\bP}.
\end{equation}
\end{remark}

\medskip

Now we turn to the proof of (\ref{tqnjncon}) and (\ref{hcDepsdel}).
\begin{proof}[Proof of Theorem \ref{Thm:2} and Theorem \ref{Thm:3}]
The convergence of of the distribution $\frac{\tq_n
-n\bJ_{\bP}}{\sqrt{n}}$ to a normal law $\cN(\textbf{0}, \bD_{\bP})$
follows directly from the central limit theorem (see Theorem
\ref{TmEDC}-c), since $\tq_n=\sum_{k=0}^{n-1}\bDelta_{\bP}\circ
T^k_{\bP}$, where $\bDelta_{\bP}$ belongs to $\cH$ by Lemma
\ref{deltataueps}. Thus it is enough to estimate the covariance
matrix $\bD_{\bP}$, which is given by the following sum of
correlations \beq\label{cDeps}
\bD_{\bP}=\sum_{n=-\infty}^{\infty}\left[
\mu_{\bP}(\bDelta_{\bP}\circ T^n_{\bP}\otimes
\bDelta_{\bP})-\mu_{\bP}(\bDelta_{\bP})\otimes
\mu_{\bP}(\bDelta_{\bP})\right]. \eeq Note that for any $n\geq 1$,
$$\lim_{\bP\to 0}\bDelta_{\bP}\circ T^n_{\bP}=\bDelta_0\circ
T^n_{0}.$$ Furthermore by Theorem \ref{TmEDC}, there exist $C>0$ and
$\theta\in (0,1)$, such that for any $\eps$-small force pair $\bP$
and for all $a,b\in\{x,y\}$,
 $$|\cC_{a,b}(n)|=
 |\mu_{\bP}(\Delta_{a,\bP}\circ T^n_{\bP}\cdot \Delta_{b,\bP})
 -\mu_{\bP}(\Delta_{a,\bP})\cdot\mu_{\bP}(\Delta_{b,\bP})|
 \leq C\theta^{|n|}.$$
Therefore the series $\sum_{n=-\infty}^{\infty} \cC_{a,b}(n)$ converges uniformly (and exponentially)
for all $a,b\in\{x,y\}$. So the diffusion matrix varies continuously  at $\bP=0$:
\begin{align}\label{cDeps2}
\bD_{\bP}&=\sum_{n=-\infty}^{\infty}
[\mu_{\bP}(\bDelta_{\bP}\circ T^n_{\bP}\otimes \bDelta_{\bP})-\mu_{\bP}(\bDelta_{\bP})
\otimes\mu_{\bP}(\bDelta_{\bP})]\nonumber\\
&=\sum_{n=-\infty}^{\infty} \mu_{0}(\bDelta_0\circ T^n_{0}\otimes
\bDelta_0)+o(1)=\bD_0+o(1),
\end{align}
since $\mu_0(\bDelta_0)=0$.
This finishes the proof of Theorem \ref{Thm:2}.  Theorem \ref{Thm:3} follows
directly from Theorem \ref{Thm:2}.
\end{proof}

Finally we prove Theorem \ref{Thm:1}.
\begin{proof}[Proof of Theorem \ref{Thm:1}]
Let $\lambda_{\bP}^s<0<\lambda^u_{\bP}$ denote the Lyapunov
exponents of the ergodic system $(T_{\bP},\mu_{\bP})$. The sum
$\xi_{\bP}:=-(\lambda_{\bP}^s+\lambda^u_{\bP})$ represents the {\it physical
entropy production rate} of the perturbed system $(T_{\bP},\mu_{\bP})$.
Then by Oseledets Multiplicity Ergodic Theorem, we have
$$\xi_{\bP}=-(\lambda_{\bP}^s+\lambda_{\bP}^u)=-\mu_{\bP} (\log \cJ_{\bP}).$$
Note that $1-\cJ_{\bP}=\cO(\epsilon)$. So we have
\begin{align*}
\xi_{\bP}&=-\mu_{\bP}(\log (1-(1-\cJ_{\bP})))
=\mu_{\bP}(1-\cJ_{\bP})\frac{1}{2}\mu_{\bP}((1-\cJ_{\bP})^2)
+\cO(\eps^3).
\end{align*}
Using the similar analysis as in (\ref{similar}) and $\mu_0(\cJ_{\bP})=1$, one can check that
\begin{align}\label{cJ1}
&\mu_{\bP}(1-\cJ_{\bP})=\mu_0(1-\cJ_{\bP})+\sum_{k=1}^{\infty} \mu_0[(1-\cJ_{\bP})\circ T_{\bP}^k \cdot (1-\cJ_{\bP})]\nonumber\\
=&\eps^2 \sum_{k=1}^{\infty} \mu_0(H\circ T_{\bP}^k \cdot H)+\cO(\eps^3)
=\eps^2 \sum_{k=1}^{\infty} \mu_0(H\circ T_{0}^k \cdot H)+o(\eps^2).
\end{align}
In addition, we have
\begin{align}\label{cJ2}
\mu_{\bP}((1-\cJ_{\bP})^2)&=\mu_0((1-\cJ_{\bP})^2)+\sum_{k=1}^{\infty} \mu_0[((1-\cJ_{\bP})^2)\circ T_{\bP}^k \cdot (1-\cJ_{\bP})]\nonumber\\
&=\eps^2 \mu_0(H^2)+\frac{\eps^3}{2}\sum_{k=1}^{\infty} \mu_0(H^2\circ T_{0}^k \cdot H)+o(\eps^3),
\end{align}
where we used (\ref{appcJ}) and (\ref{Kawasa}) in the above estimates.
Combining these facts, we get
$$\xi_{\bP}=\frac{\eps^2}{2}\mu_0(H^2)
+\eps^2\sum_{k=1}^{\infty} \mu_0(H\circ T_0^k\cdot H)+o(\eps^2)
=\eps^2\cdot \frac{\sigma^2_0(H)}{2}+o(\eps^2),$$
where
$\sigma^2_0(H)=\sum_{k=-\infty}^{\infty}\mu_0[H\circ T_0^k\cdot H]$.
Being an SRB measure, the metric entropy of $\mu_{\bP}$ satisfies
$h_{\mu_{\bP}}(T_{\bP})=\lambda_{\bP}^u=\mu_{\bP}(\Lambda^u_{\bP})$,
by Pesin's Entropy Formula and by Birkhoff Ergodic Theorem,
respectively, where $\Lambda^u_{\bP}(\bx)$ is the local expansion
rate of $\bx\in M$ along the unstable direction under map $T_{\bP}$.
Then according to (\ref{Kawasa}) and the exponential decay of
correlations, we have
\begin{align}\label{hmu}
h_{\mu_{\bP}}(T_{\bP})&=\mu_0(\Lambda^u_{\bP})
+\sum_{k=1}^{\infty} \mu_0(\Lambda^u_{\bP}\circ T^k_{\bP} \cdot (1-\cJ_{\bP}))\nonumber\\
&=\mu_0(\Lambda^u_{0})
+ \sum_{k=1}^{\infty} \mu_0(\Lambda^u_{0}\circ T_{\bP}^k \cdot  (\eps H+\eps^2 R_{\bP}))+o(1)\nonumber\\
&=h_0+o(1),
\end{align}
since $\Lambda^u_{\bP}\to \Lambda^u_{0}$ as $\bP\to 0$. Here
$h_0:=h_{\mu_0}(T_0)>0$ is the metric entropy of the unforced
billiard map $T_0$.
Combining the above facts and Young's Dimension Formula (\ref{HD})
in \cite{You}, we get
\begin{align*}
\mathrm{HD}(\mu_{\bP})&=h_{\mu_{\bP}}(T_{\bP})\left(\frac{1}{\lambda^u_{\bP}}-\frac{1}{\lambda^s_{\bP}}\right)
=1-\frac{\lambda^u_{\bP}}{\lambda^s_{\bP}}
=2-\frac{\xi_{\bP}}{h_{\mu_{\bP}}(T_{\bP})+\xi_{\bP}}\\
&=2-\frac{\eps^2\cdot\sigma^2_0(H)/2+o(\eps^2)}{h_0+\eps^2\cdot \sigma^2_0(H)/2+o(1)}\\
&=2-\eps^2\cdot\frac{\sigma^2_0(H)}{2h_0}+o(\eps^2).
\end{align*}
This completes the proof of Theorem \ref{Thm:1}.
\end{proof}
Then by the suspension property, we see that the dimension of the measure
$\hat\mu_{\bP}$ is given by
$$\mathrm{HD}(\hat\mu_{\bP})
=\mathrm{HD}(\mu_{\bP})+1 =3-\eps^2\cdot
 \frac{\sigma^2_H}{2h_0}+o(\eps^2).$$
So if $\eps$ is small and $\sigma^2_H>0$, then $1<\mathrm{HD}(\mu_{\bP})<2$ and
$2<\mathrm{HD}(\hat\mu_{\bP})<3$. Therefore, both $\mu_{\bP}$ and
$\hat \mu_{\bP}$ are singular with respect to the Lebesgue measures
and admit the fractal structures if $\sigma^2_H\neq 0$ under a small
force pair $\bP$. In fact, it is believed \cite{CELS2} (based on numerical evidences)
that $\mu_{bP}$ should be multifractal with a continuous spectrum
of fractal dimensions.

\medskip\noindent\textbf{Acknowledgement}.
N.C.\ was partially supported by National Science Foundation, grant
DMS-0969187. H.-K.Z.\  was partially supported by National Science
Foundation, grant DMS-1151762, and by the French CNRS with a {\em poste d'accueil} position  at the Center of Theoretical Physics in Luminy. P.Z. was partially supported by NNSF project (11071231).


\begin{thebibliography}{99}

\bibitem[Ba88]{Ba} P.R. Baldwin,  Soft billiard systems,
{\it Physica D} {\bf 29} (1988), 321--342.

\bibitem[BS94]{BS1} L. Bunimovich and H. Spohn, \emph{Viscosity for a period two-disk fluid: an existence proof}, Commun. Math. Phys. \textbf{ 176}, 1994, 661--680.\bibitem[BSC91]{BSC} L. A. Bunimovich, Ya. G. Sinai and N. I. Chernov,
Statistical properties of two-dimensional hyperbolic billiards,
{\it  Russ. Math. Surv.} {\bf 46} (1991), 47--106.

\bibitem[Ch99]{Ch99}
N. Chernov, Decay of correlations and dispersing billiards,
{\it J. Statist. Phys.}  {\bf 94} (1999), 513--556.


\bibitem[Ch01]{Ch01} N. Chernov,   Sinai billiards under small external forces,
  {\it   Ann. Henri Poincare} {\bf 2} (2001), 197--236.

\bibitem[Ch06]{Ch06} N. Chernov,Advanced statistical properties of
dispersing billiards,  \emph{J. Stat. Phys.} {\bf 122} (2006), 1061--1094.

\bibitem[Ch08]{Ch08} N. Chernov, Sinai billiards under small external forces II,
{\it Annales Henri Poincare}, {\bf 9} (2008), 91--107.

\bibitem[CELS93a]{CELS1} N. I. Chernov, G. L. Eyink, J. L. Lebowitz and Ya. G. Sinai,
Steady-state electrical conduction in the periodic Lorentz gas,
{\it Comm. Math. Phys.} {\bf 154}  (1993), 569--601.

\bibitem[CELS93b]{CELS2} N. I. Chernov, G. L. Eyink, J. L. Lebowitz and Ya. G. Sinai,
Derivation of Ohm's law in a deterministic mechanical model,
{\it Phys. Rev. Let.} {\bf 70} (1993), 2209--2212.

\bibitem[CK13]{CK} N. Chernov and A. Korepanov,
Spatial structure of Sinai-Ruelle-Bowen measures, {\it submitted}, 2013.

\bibitem[CM06]{CM06}  N.\ Chernov and R.\ Markarian, \emph{Chaotic Billiards}, Mathematical
Surveys and Monographs, \textbf{127}, AMS, Providence, RI, 2006.

\bibitem[CZ09]{CZ09} N. Chernov and H-K. Zhang,
Statistical properties of hyperbolic systems with general singularities,
{\it J. of Stat. Phys.}, \textbf{136} (2009),  615--642.

\bibitem[DZ]{DZ11} M.F.\ Demers and H.-K.\ Zhang,
Spectral analysis of the transfer operator for the Lorentz gas,
\emph{J. of Modern Dynamics} {\bf 5} (2011), 665--709.

\bibitem[DC09]{DC} D.  Dolgopyat and N. Chernov,
Anomalous current in periodic Lorentz gases with an infinite horizon,
(Russian) {\it Uspekhi Mat. Nauk} {\bf 64} (2009), 73--124;
translation in {\it Russian Math. Surveys} {\bf 64} (2009), 651--699.

\bibitem[G95]{G95} G. Gallavotti, {\em Reversible Anosov diffeomorphisms and large deviations},
Math. Phys. Electr. J. \textbf{1} (1995), 1--12.

\bibitem[GO74]{GO74}
G. Gallavotti and D. Ornstein, Billiards and Bernoulli schemes,
{\it  Comm. Math. Phys.} {\bf 38} (1974), 83--101.

\bibitem[Kn89]{Kn} A. Knauf,  On soft billiard systems,
{\it Physica D} {\bf 36} (1989),  259--262.

\bibitem[Lor]{Lor} H. A. Lorentz, The motion of electrons in metallic bodies,
{\it Proc. Amst. Acad.} {\bf 7} (1905), 438--453.


\bibitem[MHB87]{MHB}  B. Moran, W. Hoover and S Bestiale,
Diffusion in a periodic Lorentz gas, {\it J. of Stat. Phys.}, \textbf{48} (1987),  709--726.


\bibitem[Ru99]{Ru99} D. Ruelle,
Smooth dynamics and new theoretical ideas in nonequilibrium statistical mechanics,
{\it J. Stat. Phys.} {\bf 95} (1999), 393--468.


\bibitem[Si70]{Si70}
Ya. G. Sinai, Dynamical systems with elastic reflections. Ergodic properties of dispersing billiards,
{\it Russ. Math. Surv.} {\bf 25} (1970), 137--189.


\bibitem[You82]{You} L.S. Young, Dimension, entropy and Lyapunov exponents,
{\it Ergod. Th. and  Dynam. Sys.} \textbf{2} (1982), 109--124.

\bibitem[You98]{You98}
L.S. Young, Statistical properties of dynamical systems with some hyperbolicity,
{\it Annals of Math.} {\bf 147} (1998), 585--650.


\bibitem[Zh11]{Z11}   H.-K. Zhang,
Current in periodic Lorentz gases with twists,
{\it  Comm. Math. Phys.}   {\bf  306}  (2011), 747--776.


\end{thebibliography}
\end{document}